\documentclass[a4paper,12pt]{article}
\usepackage{amsmath, amsfonts, amssymb, amsthm}
\usepackage[english]{babel}

\newcounter{num}[section]

\newenvironment{theorem}
{\refstepcounter{num}%
\bigskip\noindent\nopagebreak[4]{\bf Theorem~\arabic{section}.\arabic{num}. }\it}

\newenvironment{proposition}
{\refstepcounter{num}%
\bigskip\noindent\nopagebreak[4]{\bf Proposition~\arabic{section}.\arabic{num}. }\it}

\newenvironment{corollary}
{\refstepcounter{num}%
\bigskip\noindent\nopagebreak[4]{\bf Corollary~\arabic{section}.\arabic{num}. }\it}

\newenvironment{lemma}
{\refstepcounter{num}%
\bigskip\noindent\nopagebreak[4]{\bf Lemma~\arabic{section}.\arabic{num}. }\it}

\newenvironment{example}
{\refstepcounter{num}%
\bigskip\noindent\nopagebreak[4]{\bf Example~\arabic{section}.\arabic{num}. }}

\newcommand{\LL}{{\mathcal{L}}}

\newcommand{\Ss}{{\mathbf{S}}}
\newcommand{\V}{{\mathrm{V}}}
\newcommand{\supp}{{\mathrm{supp}}}
\newcommand{\Ann}{{\mathrm{Ann}}}

\newcommand{\om}{{\omega}}
\newcommand{\pr}{{\prime}}

\newcommand{\al}{{\alpha}}

\newcommand{\g}{{\mathbf{g}}}
\newcommand{\h}{{\mathbf{h}}}
\renewcommand{\r}{{\mathbf{r}}}

\newcommand{\Vbf}{{\mathbf{V}}}

\newcommand{\one}{{\mathbf{1}}}
\newcommand{\zero}{{\mathbf{0}}}

\newcommand{\qq}{{\mathbf{q}_\om}}

\begin{document}

\title{Direct products, varieties, and compactness conditions}
\author{M. Shahryari, A. N. Shevlyakov}
\maketitle
\abstract{We study equationally Noetherian varieties of groups, rings and monoids. Moreover, we describe equationally Noetherian direct powers for these algebraic structures.}
\section{Introduction}

A variety is one of significant notions in universal algebraic geometry. Let us explain its importance. Let $\LL$ be a language. An $\LL$-equation is an atomic formula $t(X)=s(X)$, where $t,s$ are $\LL$-terms. Let $A$ be an algebraic structure of $\LL$ ($\LL$-algebra). Suppose we are dealing with $\LL$-equations over $A$, and $A$ belongs to a certain variety $\Vbf$ of $\LL$-algebras. Then we can easily simplify $\LL$-terms (parts of equations) using the identities of $\Vbf$. For example, any term $t(X)$ over a semigroup $S$ admits an elimination of brackets, since $S$ satisfies the associativity law (identity) $x(yz)=(xy)z$. Similarly, any ring with zero multiplication admits a reduction of any term to a multiplication-free expression.

Another important notion in universal algebraic geometry is the property of being equationally Noetherian. Recall that an algebraic structure $A$ is  $\LL$-equationally Noetherian if any system of $\LL$-equations is equivalent over $A$ to its finite subsystem. The problem about the connections between varieties and equationally Noetherian algebras was posed by B.~Plotkin in~\cite{Plotkin}.

\bigskip

\noindent {\bf Problem 1.} Is there a variety $\Vbf$ of $\LL$-algebras such that every $A\in\Vbf$ is  $\LL$-equationally Noetherian algebra?

\bigskip

This problem has positive solutions for many varieties. For example, in~\cite{ModSH} it was proved that all elements of a group variety $\mathbf{V}$ are $\LL_g$-equationally Noetherian ($\LL_g=\{\cdot,{ }^{-1},1\}$ is the group language), if and only if the free group $F_{\mathbf{V}}(X)\in\Vbf$ has the $\mathrm{max-n}$ property for every finite set $X$. It follows that the variety of all metabelian groups satisfies Problem~1.

Let us formulate the central problem of the current paper (it slightly differs from the original Plotkin`s problem).

\bigskip

\noindent {\bf Problem 2.} Is there a variety $\Vbf$ of $\LL$-algebras such that each $A\in\Vbf$ is  $\LL(A)$-equationally Noetherian algebras?

\bigskip

The language $\LL(A)=\LL\cup \{a\mid a\in A\}$ above is the extension of $\LL$ by new constant symbols $a\in A$ corresponding to all elements of $A$. Recall that $\LL(A)$ defines the wider class of equations than $\LL$, because it allows to use constants in equations. Thus, an $\LL$-equationally Noetherian algebra $A$ is not necessarily $\LL(A)$-equationally Noetherian.

We prove (Theorems~\ref{th:group_var_from_N},\ref{th:rings_var_from_N}) that a variety of groups (rings) satisfies Problem~2, if and only if it is abelian (respectively, has zero multiplication). The similar result holds for monoids (Theorem~\ref{th:monoid_var_from_N}).

The obtained results are based on equational properties of direct powers of groups (Theorem~\ref{th:group_restricted_power_is_not_q}) rings (Theorem~\ref{th:ring_restricted_power_is_not_q}) and monoids (Theorem~\ref{th:monoid_restricted_power_is_not_q}). Moreover, in Theorem~\ref{th:magma_restricted_power_is_not_q} we consider direct powers of arbitrary algebras with binary operation (magmas). Actually, in such theorems we prove that an infinite direct power of a certain group (ring, monoid, magma) is not $\qq$-compact.

Let us mention an application of our results. In~\cite{BMRom} it was proved that any wreath product (restricted and unrestricted) $W=A\mathrm{wr}B$ of a non-abelian group $A$ and infinite group $B$ is not $\LL(W)$-equationally Noetherian. Theorem~\ref{th:group_restricted_power_is_not_q} allows us to obtain a simple proof of this fact.

In Section~\ref{sec:var_of_groups_in_L_g} we deal with the language $\LL_g$ (the language of groups), and we study axiomatic ranks and finite basis problem for group varieties $\Vbf$ such that each $G\in\Vbf$ is $\LL_g$-equationally Noetherian.

\section{Preliminaries}

Let us give some basic notions of universal algebraic geometry following~\cite{DMR1}--~\cite{DMR3}. Let $\LL$ be a functional language. In the current paper we consider as special cases, languages of the following types: $\LL_m=\{\cdot,1\}$ (monoid language), $\LL_g=\{\cdot,{}^{-1},1\}$ (group language), $\LL_r=\{+,-,\cdot,0\}$ (ring language).

Let $A$ be an algebraic structure of the language $\LL$ ($\LL$-algebra). An {\it equation over  $\LL$ ($\LL$-equation)} is an equality of two terms of $\LL$:
$$
p(X)=q(X).
$$

The examples of equations in various languages are: $[x,y]=1$ (here $[x,y]=x^{-1}y^{-1}xy$), $x^{-1}y^3x=y^2$ (language $\LL_g$); $xy=yx$, $x^2=1$ (language $\LL_m$);  $x^2+y^2=z^2$, $xy+xy+yz=0$ (language $\LL_r$).

A system of $\LL$-equations ($\LL$-system for shortness) is an arbitrary set $\Ss$ of $\LL$-equations. {\it Notice that we will consider only systems which depend on a finite set of variables $X=\{x_1,x_2,\ldots,x_n\}$}. The set of all solutions of $\Ss$ in $A$ is denoted by $\V_A(\Ss)\subseteq A^n$.

An $\LL$-algebra $A$ is {\it $\LL$-equationally Noetherian} if any infinite $\LL$-system $\Ss$ is equivalent over $A$ to some finite subsystem $\Ss^\pr\subseteq \Ss$. An $\LL$-algebra $A$ is {\it $\qq$-compact} in the language $\LL$ if for any infinite $\LL$-system $\Ss$ and an $\LL$-equation $p(X)=q(X)$ such that
$$
\V_A(\Ss)\subseteq \V_A(p(X)=q(X))
$$
there exists a finite subsystem $\Ss^\pr\subseteq\Ss$ with
$$
\V_A(\Ss^\pr)\subseteq \V_A(p(X)=q(X)).
$$
According to the definitions, any $\LL$-equationally Noetherian algebra is $\qq$-compact in $\LL$.

Let $A$ be an $\LL$-algebra. By $\LL(A)$ we denote the language $\LL\cup\{a\mid a\in A\}$ extended by new constants symbols which correspond to elements of $A$. The language extension allows us to use constants in equations. The examples of equations in extended languages are the following: $[x,g]=1$ (language $\LL_g(G)$, $g\in G$); $xm=mx$ (language $\LL_m(M)$, $m\in M$); $r\cdot x^2+r^\pr\cdot y^2=0$ (language $\LL_r(R)$, $r,r^\pr\in R$). Obviously, the class of $\LL(S)$-equations is wider than the class of $\LL$-equations, so an $\LL$-equationally Noetherian algebra may lose this property in the language $\LL(A)$. For simplicity, we will say that $A$ is {\it equationally Noetherian} if $A$ is $\LL(A)$-equationally Noetherian and $A$ is {\it $\qq$-compact} is it is $\qq$-compact in the language $\LL(A)$.

The following proposition contains simple examples of equationally Noetherian algebraic structures.

\begin{proposition}
\begin{enumerate}
\item Any abelian group $G$ is equationally Noetherian.
\item Any ring $R$ with zero multiplication is equationally Noetherian.
\end{enumerate}
\label{pr:examples_of_noeth}
\end{proposition}
\begin{proof}
The first statement is well-known in universal algebraic geometry (see~\cite{DMR1,DMR2}).

Let $R$ be a ring with zero multiplication, i.e. $ab=0$ for all $a,b\in R$. Therefore, any $\LL_r(R)$-equation is equivalent to one of the following simple expressions $x_i=x_j$, $x_i=a$, $a=b$, where $a,b\in R$. So any system $\Ss$ does not contain an infinite number of pairwise non-equivalent $\LL_r(R)$-equations. Thus, $\Ss$ should be equivalent to its finite subsystem.
\end{proof}

Let $A$ be an $\LL$-algebra. The infinite \textit{direct power} $A^{\infty}=\prod_{i=1}^\infty A$  is an $\LL$-algebra of all sequences
\[
(a_1,a_2,\ldots,a_n,\ldots),\; a_i\in A,
\]
and the definition of each function $f\in\LL$ over $A^{\infty}$ is element-wise.

Let $G^\infty$ be the direct power of a group (monoid) $G$, and $\g=(g_1,g_2,\ldots)\in G^\infty$. Then the {\it support} $\supp(\g)$ is the set of indexes with $g_i\neq 1$. The support $\supp(\r)$ of an element $\r=(r_1,r_2,\ldots)$ in the direct power $R^\infty$ of a ring $R$ is defined as $\{i\mid r_i\neq 0\}$.

The {\it restricted direct power} $G^{(\infty)}$ of a group $G$ is a subgroup in $G^\infty$ of all elements $\g$ with $|\supp(\g)|<\infty$. Similarly, one can define the  restricted direct power for monoids and rings.

\section{Direct products}

\begin{theorem}
Let $G$ be a non-abelian group. Then the restricted direct power $G^{(\infty)}$ is not $\qq$-compact.
\label{th:group_restricted_power_is_not_q}
\end{theorem}

\begin{proof}
Let $a\in G$, and $\Ss_a(x)$ be an infinite system of $G^{(\infty)}$-equations defined as follows:
$$
\Ss_a(x)=
\begin{cases}
[x,(a,1,1,1\ldots)]=\one,\\
[x,(a,a,1,1\ldots)]=\one,\\
[x,(a,a,a,1\ldots)]=\one,\\
\ldots\\
\end{cases}
$$
Denote $\Ss(x)=\bigcup_{a\in G}\Ss_a$. Obviously,
$$
\V_{G^{(\infty)}}(\Ss(x))=\{\g=(g_1,g_2,g_3\ldots)\mid g_i\in Z(G), \supp(\g)<\infty \}.
$$
and therefore $\V_{G^{(\infty)}}(\Ss(x))=Z(G^{(\infty)})$ (here $Z(G)$ is the center of $G$).

Let $\Ss(x,y)=\Ss(x)\cup\Ss(y)$ ($\Ss(y)$ is a clone of $\Ss(x)$, where all occurrences of $x$ are replaced to $y$). Then
$$
\V_{G^{(\infty)}}(\Ss(x,y))=\{(\g,\h)\mid \g,\h\in Z(G^{(\infty)})\}.
$$
and the following inclusion
\begin{equation}
\V_{G^{(\infty)}}(\Ss(x,y))\subseteq\V_{G^{(\infty)}}([x,y]=\one)
\end{equation}
holds.

Let us consider a finite subsystem $\Ss^\pr(x,y)$ of $\Ss(x,y)$. Without loss of generality one can assume that
$$
\Ss^\pr(x,y)=\bigcup_{a\in T} \left(\Ss_a^\pr(x)\cup\Ss_a^\pr(y)\right),
$$
where $T$ is a finite subset of $G$ and $\Ss_a^\pr(x),\Ss_a^\pr(y)$ are the first $n$ equations of the systems $\Ss_a(x),\Ss_a(y)$ respectively.

Since $G$ is not abelian, there exists a pair $\al,\beta\in G$ with $[\al,\beta]\neq 1$. Let us consider the following elements
$$
\g=(\underbrace{1,1,\ldots,1}_{\mbox{$n$ times}},\al,1,1,\ldots)\in G^{(\infty)},
$$
$$
\h=(\underbrace{1,1,\ldots,1}_{\mbox{$n$ times}},\beta,1,1,\ldots)\in G^{(\infty)}.
$$
One can directly check that $(\g,\h)\in\V_{G^{(\infty)}}(\Ss^\pr(x,y))$, but $[\g,\h]\neq 1$. Therefore, the inclusion
\begin{equation}
\V_{G^{(\infty)}}(\Ss^\pr(x,y))\subseteq\V_{G^{(\infty)}}([x,y]=\one)
\end{equation}
fails. Thus, $G^{(\infty)}$ is not $\qq$-compact.

\end{proof}

Since $G^{(\infty)}$ is embedded into $G^\infty$, we immediately obtain the following result.

\begin{corollary}
The direct power $G^{\infty}$ is not $\qq$-compact for any non-abelian group $G$.
\label{cor:th:group_power_is_not_N}
\end{corollary}

\bigskip

In \cite{BMRom}, it is proved that if $G$ is non-abelian and $A$ is infinite group, then the wreath product $W=G\mathrm{wr}A$ (restricted or unrestricted) is not equationally Noetherian. Now, we see that this is an instant corollary of Theorem~\ref{th:group_restricted_power_is_not_q}, since $G^{(\infty)}$ is embedded into $W$.

\begin{corollary}
Let $G\mathrm{wr}A$  be the  wreath product of a non-abelian group $G$ and infinite group $A$. Then this group is not equationally Noetherian.
\end{corollary}

\bigskip

Let us study ring equations.

\begin{theorem}
Let $R$ be a ring with a nonzero multiplication. Then the restricted direct power $R^{(\infty)}$ is not $\qq$-compact.
\label{th:ring_restricted_power_is_not_q}
\end{theorem}
\begin{proof}
Let us consider an infinite system
$\Ss_a(x)$  of $R^{(\infty)}$-equations defined as follows:
$$
\Ss_a(x)=
\begin{cases}
(a,0,0,0,\ldots)\cdot x=\zero,\\
(a,a,0,0,\ldots)\cdot x=\zero,\\
(a,a,a,0,\ldots)\cdot x=\zero,\\
\ldots
\end{cases}
$$

Denote $\Ss(x)=\bigcup_{a\in R}\Ss_a$. Obviously,
$$
\V_{R^{(\infty)}}(\Ss(x))=\{\r=(r_1,r_2,r_3\ldots)\mid r_i\in \Ann_r(R), \supp(\r)<\infty \},
$$
where
$$
\Ann_r(R)=\{b\in R\mid ab=0\mbox{ for all $a\in R$}\}
$$
is the right annihilator of $R$. Therefore $\V_{R^{(\infty)}}(\Ss(x))=\Ann_r(R^{(\infty)})$.

Let $\Ss(x,y)=\Ss(x)\cup\Ss(y)$ ($\Ss(y)$ is a clone of $\Ss(x)$, where all occurrences of $x$ are replaced to $y$). Then
\[
\V_{R^{(\infty)}}(\Ss(x,y))=\{(\r,\r^\pr)\mid \r,\r^\pr\in \Ann_r(R^{(\infty)})\}.
\]
and the following inclusion
\begin{equation}
\V_{G^{(\infty)}}(\Ss(x,y))\subseteq\V_{G^{(\infty)}}(xy=\zero)
\end{equation}
holds.

Let us consider a finite subsystem $\Ss^\pr(x,y)$ of $\Ss(x,y)$. Without loss of generality one can assume that
\[
\Ss^\pr(x,y)=\bigcup_{a\in T} \left(\Ss_a^\pr(x)\cup\Ss_a^\pr(y)\right),
\]
where $T$ is a finite subset of $R$ and $\Ss_a^\pr(x),\Ss_a^\pr(y)$ are the first $n$ equations of the systems $\Ss_a(x),\Ss_a(y)$ respectively.

By the condition, there exists a pair $\al,\beta\in R$ with $\al\beta\neq 0$. Let us consider the following elements
\[
\r=(\underbrace{0,0,\ldots,0}_{\mbox{$n$ times}},\al,0,0,\ldots)\in R^{(\infty)},
\]
\[
\r^\pr=(\underbrace{0,0,\ldots,0}_{\mbox{$n$ times}},\beta,0,0,\ldots)\in R^{(\infty)}.
\]
One can directly check that $(\r,\r^\pr)\in\V_{R^{(\infty)}}(\Ss^\pr(x,y))$, but $\r\r^\pr\neq 0$. Therefore, the inclusion
\begin{equation}
\V_{R^{(\infty)}}(\Ss^\pr(x,y))\subseteq\V_{R^{(\infty)}}(xy=\zero)
\end{equation}
fails. Thus, $R^{(\infty)}$ is not $\qq$-compact.

\end{proof}

We have the analogue of Corollary~\ref{cor:th:group_power_is_not_N} for rings.

\begin{corollary}
Let $R$ be a ring with non-zero multiplication. Then the direct power $R^{\infty}$ is not $\qq$-compact.
\label{cor:th:ring_power_is_not_N}
\end{corollary}

\bigskip

Also Theorem~\ref{th:group_restricted_power_is_not_q} allows to prove the following statement.

\begin{theorem}
Let $M$ be a non-commutative monoid. Then the restricted direct power $M^{(\infty)}$ is not $\qq$-compact.
\label{th:monoid_restricted_power_is_not_q}
\end{theorem}
\begin{proof}
One should rewrite a system $\Ss_a$ from Theorem~\ref{th:group_restricted_power_is_not_q} as a system of $\LL_m$-equations.
\[
\Ss_a(x)=
\begin{cases}
x(a,1,1,1\ldots)=(a,1,1,1\ldots)x,\\
x(a,a,1,1\ldots)=(a,a,1,1\ldots)x,\\
x(a,a,a,1\ldots)=(a,a,a,1\ldots)x,\\
\ldots\\
\end{cases}
\]
and the rest proof is similar to the argument of Theorem~\ref{th:group_restricted_power_is_not_q}.
\end{proof}

One can generalize the above arguments even for wider class of algebras. A {\it magma} is an arbitrary algebraic structure of the language $\LL_s=\{\cdot\}$. The natural example of a magma is the following.

\begin{example}
Let $\LL$ be an arbitrary language. If $A$ is an $\LL$-algebra and $p(x,y)$ is an $\LL$-term, then the the set $A$ equipped with the binary operation $p(x,y)$ is a magma. Let us denote the obtained magma by $(A,p)$
\label{ex:magma}
\end{example}

\bigskip

One can naturally give the definitions of the commutativity and  center for magmas, and the proof of the following result is similar to Theorem~\ref{th:monoid_restricted_power_is_not_q}.

\begin{theorem}
Let $M$ be a non-commutative magma with non-empty center. Then the direct power $M^{\infty}$ is not $\qq$-compact.
\label{th:magma_restricted_power_is_not_q}
\end{theorem}

\bigskip

Theorem~\ref{th:magma_restricted_power_is_not_q} and Example~\ref{ex:magma} give a result about varieties of $\LL$-algebras.

\begin{corollary}
If every element $A$ of a variety $\mathbf{V}$ of $\LL$-algebras is equationally Noetherian, then for any $\LL$-term $p(x,y)$, the magma $(A, p)$ is commutative or has empty center.
\end{corollary}

\bigskip

The results of Theorems~\ref{th:group_restricted_power_is_not_q},~\ref{th:ring_restricted_power_is_not_q},
~\ref{th:monoid_restricted_power_is_not_q} shows that an infinite direct product is not necessarily equationally Noetherian. However, the finite direct products of equationally Noetherian algebraic structures are equationally Noetherian.


\begin{proposition}
Suppose that $A$ and $B$ are two $\LL$-algebras. Let $A$ be $\LL(A)$-equationally Noetherian and $B$ be $\LL(B)$-equationally Noetherian. Then $C=A\times B$ is $\LL(C)$-equationally Noetherian.
\end{proposition}

\begin{proof}
First, let $p(X)=p(x_1, \ldots, x_n, (a_1, b_1), \ldots, (a_m, b_m))$ be a term in the language $\LL(C)$. Note that, here all $(a_i, b_i)$ are coefficients from $C$. There are two new terms corresponding to $p(X)$; the first one is a term in the language $\LL(A)$ which we denote it by $p_A$ and the second one is a term in the language $\LL(B)$ which will be denoted by $p_B$. The term $p_A$ is obtained from $p(X)$ by replacing every variable $x_i$ by a new variable $y_i$ and every coefficient $(a_i, b_i)$ by $a_i$. The same is true also for $p_B$, it is obtained from $p(X)$ by replacing every variable $x_i$ by a new variable $z_i$ and every coefficient $(a_i, b_i)$ by $b_i$. Let $p(X)=q(X)$ be an equation in the language $\LL(C)$. Let $((\alpha_1, \beta_1), \ldots, (\alpha_n, \beta_n))\in C^n$ be a solution of this equation. Then we have
$$
(p_A(\alpha_1, \ldots, \alpha_n, \overline{a}), p_B(\beta_1, \ldots, \beta_n, \overline{b}))=(q_A(\alpha_1, \ldots, \alpha_n, \overline{a}), q_B(\beta_1, \ldots, \beta_n, \overline{b})),
$$
where $\overline{a}=(a_1,\ldots,a_m)$, $\overline{b}=(b_1,\ldots,b_m)$.
So $(\alpha_1, \ldots, \alpha_n)$ is a solution of $p_A(Y)=q_A(Y)$ and $(\beta_1, \ldots, \beta_n)$ is a solution of $p_B(Z)=q_B(Z)$. The converse is also true. In general, let $\Ss$ be a system of equations in the language $\LL(C)$. Let
$$
\Ss^A=\{ p_A(Y)=q_A(Y):\ (p(X)=q(X))\in \Ss\}
$$
and similarly define $\Ss^B$. An easy argument shows that there is a bijection between the algebraic set $\V_C(\Ss)$ and the Cartesian product $\V_A(\Ss^A)\times \V_B(\Ss^B)$. Now, let
$$
\V_C(\Ss_1)\supseteq \V_C(\Ss_2)\supseteq \V_C(\Ss_3)\supseteq \cdots
$$
be a descending chain of algebraic sets in the space $C^n$. Correspondingly, we have a chain
$$
\V_A(\Ss_1^A)\times \V_B(\Ss_1^B)\supseteq \V_A(\Ss_2^A)\times \V_B(\Ss_2^B)\supseteq \V_A(\Ss_3^A)\times \V_B(\Ss_3^B)\supseteq\cdots,
$$
and hence, two partial chains
$$
\V_A(\Ss_1^A)\supseteq \V_A(\Ss_2^A)\supseteq \V_A(\Ss_3^A)\supseteq \cdots,
$$
$$
\V_B(\Ss_1^B)\supseteq \V_B(\Ss_2^B)\supseteq \V_B(\Ss_3^B)\supseteq \cdots.
$$
Since $A$ and $B$ are equationally Noetherian in their own languages, there exists an integer $k$ such that
$$
\V_A(\Ss_k^A)=\V_A(\Ss_{k+1}^A)=\cdots, \ \ \V_B(\Ss_k^B)=\V_B(\Ss_{k+1}^B)=\cdots,
$$
and hence, we have
$$
\V_C(\Ss_k)=\V_C(\Ss_{k+1})=\cdots.
$$
This shows that $C$ is $\LL(C)$-equationally Noetherian.
\end{proof}

\section{Equationally Noetherian Varieties}

In this section, we solve Problem~2 for varieties of groups, rings and monoids. Actually, we show  more generally that if a class of groups (rings) is closed under the restricted direct power, the all elements of this class are either abelian (respectively, with zero multiplication) or  at least one element of the class is not equationally Noetherian.

\begin{theorem}
Let $\Vbf$ be a variety of groups. Each $G\in \Vbf$ is equationally Noetherian, if and only if $\Vbf$ is abelian.
\label{th:group_var_from_N}
\end{theorem}
\begin{proof}
The ``if'' part of the statement follows from Proposition~\ref{pr:examples_of_noeth}. Let us prove the ``only if'' part.

Assume there exists a non-abelian group $G\in \Vbf$. Since any variety is closed under direct products, we have $G^\infty\in\Vbf$. By Theorem~\ref{th:group_restricted_power_is_not_q}, $G^\infty$ is not equationally Noetherian, and we obtain the contradiction.
\end{proof}

The theorem above can be used for the case of linear groups. Recall that a group $G$ is linear, if there exists a Noetherian ring $R$ and a natural number $n$, such that $G$ embeds in $GL_n(R)$. It is well-known that every linear group is equationally Noetherian. So, we have

\begin{corollary}
If a class of linear groups is closed under restricted or unrestricted direct power, then all elements of that class are abelian.
\end{corollary}

\begin{theorem}
Let $\Vbf$ be a variety  of rings. Each $R\in \Vbf$ is equationally Noetherian, if and only if all elements of $\Vbf$ have  zero multiplication.
\label{th:rings_var_from_N}
\end{theorem}
\begin{proof}
The ``only if'' part of the statement directly follows from Theorem~\ref{th:ring_restricted_power_is_not_q}. The ``if'' part follows from Proposition~\ref{pr:examples_of_noeth}.
\end{proof}

Now we describe monoid varieties $\Vbf$, where each $M\in\Vbf$ is equationally Noetherian.

\begin{lemma}
Let $L_2=\{0,1\}$ be a two-element semilattice ($0\cdot 0=0\cdot 1=1\cdot 0=0$, $1\cdot 1=1$). Then the direct power $L_2^\infty$ is not equationally Noetherian.
\end{lemma}
\begin{proof}
Let us consider an infinite system of  $L_2^{\infty}$-equations
\begin{equation}
\label{eq:system_S_for_L_2}
\Ss=
\begin{cases}
x\cdot(1,1,1,\ldots)=x,\\
x\cdot(0,1,1,\ldots)=x,\\
x\cdot(0,0,1,\ldots)=x,\\
\ldots
\end{cases}
\end{equation}

We have $\V_{L_2^\infty}(\Ss)=\{(0,0,0,\ldots)\}$ . However, the solution set of the first $n$ equations of $\Ss$ is
\[
\{(\underbrace{0,0,\ldots,0}_{\mbox{$n$ times}},a_{n+1},a_{n+2},\ldots)\mid a_i\in L_2\}.
\]

Thus, $\Ss$ is not equivalent to its finite subsystems, so $L_2^{\infty}$ is not equationally Noetherian.
\end{proof}

\begin{corollary}
Let $\Vbf$ be a variety of monoids such that each $M\in \Vbf$ is equationally Noetherian. Then $L_2\notin \Vbf$.
\label{cor:L_2_notin_V}
\end{corollary}

\begin{theorem}
Let $\Vbf$ be a variety of monoids. Each $M\in \Vbf$ is equationally Noetherian, if and only if $\Vbf$ is a variety of abelian groups defined by the identity $x^n=1$.
\label{th:monoid_var_from_N}
\end{theorem}
\begin{proof}
The ``if'' part of the statement follows from Proposition~\ref{pr:examples_of_noeth}. Let us prove the ``only if'' part.

Theorem~\ref{th:monoid_restricted_power_is_not_q} immediately gives that $\Vbf$ is an abelian variety.

By Corollary~\ref{cor:L_2_notin_V}, there exists an identity $p(X)=q(X)$ such that $p(X)=q(X)$ is true in any $M\in\Vbf$ and $p(X)=q(X)$ does not hold in $L_2$. According to the properties of $L_2$, there exists a variable $x$ occurring in $p(X)$ not in $q(X)$. Let us substitute all variables (except $x$) in $p(X)=q(X)$ to $1$, and we obtain an identity $x^n=1$ which holds in $\Vbf$. Thus, any $M\in\Vbf$ is a group (the inverse $a^{-1}$ of any $a\in M$ is $a^{n-1}$).
\end{proof}

\section{$\LL_g$-equations and varieties of groups}
\label{sec:var_of_groups_in_L_g}

Although the variety of abelian groups is a unique example where every element $G$ is equationally Noetherian, there are many examples of varieties every group of which is $\LL_g$-equationally Noetherian (1-equationally Noetherian for shortness). We can use results of~\cite{ModSH} to find such examples. Let us formulate the result of~\cite{ModSH} for groups of language $\LL_g$ (below $\mathrm{max-n}$ is the following property of a group $G$: every ascending chain of normal subgroups $H\vartriangleleft G$ becomes stationary).

\begin{theorem}\textup{\cite{ModSH}}
All groups of a variety $\mathbf{V}$ are 1-equationally Noetherian, if and only if the free group $F_{\mathbf{V}}(X)\in\Vbf$ has the $\mathrm{max-n}$ property  for every finite set $X$.
\label{th:noeth_iff_max-n}
\end{theorem}

\bigskip

Since any finitely generated metabelian group has the $\mathrm{max-n}$ property, all metabelian groups are 1-equationally Noetherian. It is also easy to see that every finitely generated nilpotent group has the maximal condition on its subgroups so applying Theorem~\ref{th:noeth_iff_max-n}, we see that every element of a nilpotent variety is 1-equationally Noetherian. Below, we show that there is a close connection between such property of varieties and the property of being finitely based.

Suppose $\mathbf{V}$ is a variety of groups and $\mathbf{W}$ is a subvariety. We say that $\mathbf{W}$ is relatively finitely based, if it can be defined by a finite number of identities inside $\mathbf{V}$. The relative axiomatic rank of $\mathbf{W}$ is finite, if it can be defined using finite number of variables inside $\mathbf{V}$. We apply Theorem~\ref{th:noeth_iff_max-n}, to prove the following result.

\begin{theorem}
Let $\mathbf{V}$ be a variety of groups such that its all elements are 1-equationally Noetherian. Let $\mathbf{W}$ be a subvariety with the finite relative axiomatic rank. Then $\mathbf{W}$ is relatively finitely based.
\end{theorem}

\begin{proof}
Let $R=\mathrm{id}^{\mathbf{V}}_X(\mathbf{W})$ be the set of all $\mathbf{V}$-identities defining $\mathbf{W}$ inside $\mathbf{V}$. We can assume that $X$ is finite.  Clearly we have $R\unlhd F_{\mathbf{V}}(X)$. Since every element of $\mathbf{V}$ is 1-equationally Noetherian, so by Theorem 1.1, the group $F_{\mathbf{V}}(X)$ has $\mathrm{max-n}$ and hence $R$ is finitely generated as a normal subgroup. So, there are elements $u_1, \ldots, u_m\in F_{\mathbf{V}}(X)$, such that $R$ is the normal closure of $\{u_1, \ldots, u_m\}$. Suppose
$$
\mathbf{W}^{\prime}=\{ G\in \mathbf{V}: G\models (u_1=1), \ldots, (u_m=1)\}.
$$
We show that $\mathbf{W}^{\prime}=\mathbf{W}$. Clearly $\mathbf{W}\subseteq \mathbf{W}^{\prime}$. Let $G\in \mathbf{W}^{\prime}$ and $v\in R$. Then $v=\prod w_iu_{t_i}^{\pm1}w_i^{-1}$ for some elements $w_1, w_2, \ldots$ and $u_{t_1}, u_{t_2}, \ldots$. Therefore $G\models (v=1)$ and hence $G\in \mathbf{W}$. This proves that $\mathbf{W}$ is relatively finitely based.
\end{proof}

The information of the authors:

Mohammad Shahryari

Department of Pure Mathematics

Faculty of Mathematical Sciences

University of Tabriz

Tabriz, Iran.

e-mail: \texttt{mshahryari@tabrizu.ac.ir}\\

Artem N. Shevlyakov

Sobolev Institute of Mathematics

644099 Russia, Omsk, Pevtsova st. 13

\medskip

Omsk State Technical University

pr. Mira, 11, 644050

Phone: +7-3812-23-25-51.

e-mail: \texttt{a\_shevl@mail.ru}


\bigskip
\begin{thebibliography}{1}


\bibitem{BMR} Baumslag G., Myasnikov A.,  Remeslennikov V.
{\it Algebraic geometry over groups, I. Algebraic sets and ideal theory}. J.
 Algebra, {219} (1999), 16--79.

\bibitem{BMRom}
Baumslag G., Myasnikov A.,  Romankov V. {\it Two theorems about equationally
Noetherian groups}. J. Algebra, 194 (1997), 654--664.

\bibitem{DMR1}
Daniyarova E., Myasnikov A., Remeslennikov V. {\it Unification theorems in algebraic geometry }. Algebra and Discrete Mathamatics, 1 (2008), 80--112.

\bibitem{DMR2}
Daniyarova E., Myasnikov A., Remeslennikov V. {\it Algebraic geometry over algebraic structures, II: Fundations}. J. Math. Sci., 185:3 (2012), 389--416.

\bibitem{DMR3}
Daniyarova E., Myasnikov A., Remeslennikov V. {\it Algebraic geometry over algebraic structures, III: Equationally noetherian property and
compactness}. South. Asian Bull. Math., 35:1 (2011), 35--68.


\bibitem{ModSH}
Modabberi, P. Shahryari, M.  {\it Compactness  conditions in universal algebraic geometry}.
Algebra and Logic, 55:2 (2016), 146--172.

\bibitem{Plotkin}
B. I. Plotkin, Problems in algebra inspired by universal algebraic geometry, Fundam. Prikl. Mat., 10:3 (2004), 181--197




\end{thebibliography}
\end{document}